\documentclass{article}
\usepackage{amsmath, amsthm, amsfonts, bm, bbm, caption, diagbox, tabularx, tabularray, tikz-cd, hyperref}
\usepackage[margin=1in]{geometry}

\makeatletter
\def\smalloverbrace#1{\mathop{\vbox{\m@th\ialign{##\crcr\noalign{\kern3\p@}%
\tiny\downbracefill\crcr\noalign{\kern3\p@\nointerlineskip}%
$\hfil\displaystyle{#1}\hfil$\crcr}}}\limits}
\makeatother

\newcommand\tikzmark[2]{%
\tikz[remember picture,baseline] \node[inner sep=2pt,outer sep=0] (#1){#2};%
}

\newcommand\link[2]{%
\begin{tikzpicture}[remember picture, overlay, >=stealth, shift={(0,0)}]
\draw[-{Implies},double, thick] (#1) to (#2);
\end{tikzpicture}%
}

\newcommand{\ffrac}[2]{\ensuremath{\frac{\displaystyle #1}{\displaystyle #2}}}

\newcommand{\lb}{\left[}           % left bracket 
\newcommand{\rb}{\right]}          % right bracke
\newcommand{\ev}[1]{( #1 )} 
\newcommand{\ew}[1]{\left( #1 \right)} 
\newcommand{\Mod}[1]{\ (\mathrm{mod}\ #1)}

\newcommand{\Z}{\mathbb{Z}}
\newcommand{\C}{\mathbb{C}}

\newcommand{\R}{\mathbb{R}}
\newtheorem{theorem}{Theorem}[section]
\newtheorem{lemma}[theorem]{Lemma}
\newtheorem{corollary}[theorem]{Corollary}
\newtheorem{remark}{Remark}
\newtheorem{conj}[theorem]{Conjecture}
\numberwithin{equation}{section}
\theoremstyle{definition}
\newtheorem{defn}{Definition}

\title{Polynomial Reconstruction Problem for Hypergraphs}
\author{Joshua Cooper, Utku Okur}
\date{\today}

\begin{document}
\maketitle

\begin{abstract}
We show that, in general, the characteristic polynomial of a hypergraph is not determined by its ``polynomial deck'', the multiset of characteristic polynomials of its vertex-deleted subgraphs, thus settling the ``polynomial reconstruction problem'' for hypergraphs in the negative.  The proof proceeds by showing that a construction due to Kocay of an infinite family of pairs of $3$-uniform hypergraphs which are non-isomorphic but share the same hypergraph deck, in fact, have different characteristic polynomials.  The question remain unresolved for ordinary graphs.
\end{abstract}

\section{Introduction}
The graph reconstruction conjecture, which remains open today, was first stated in \cite[p.29]{ulam} as a metric space problem. In graph theoretical terms, the \textit{vertex deck} of a graph $G = (V(G),E(G))$ is the multiset of isomorphism classes of vertex-deleted induced subgraphs $G-u$, for each $u\in V\ev{G}$. We may restate the reconstruction conjecture as follows:
\begin{conj}[Kelly, 1957; Ulam, 1960]
A graph $G$ with at least three vertices is uniquely determined, up to isomorphism, from its vertex deck. %IMPROVED
\end{conj}
Kelly has shown in \cite{kelly} that the reconstruction conjecture holds for trees. Recently, triangle-free graphs with some restrictions on the connectivity and the diameter are shown to be reconstructible from their vertex deck. (\cite{clifton}) Other reconstruction problems have been proposed: 
\begin{itemize}
\item Reconstruction of a graph from its vertex deck without multiplicity
\item Reconstruction from edge-deleted subgraphs
\end{itemize}
In \cite{schwenk} and \cite{cvetkovic}, the \textit{polynomial reconstruction problem} was suggested: Can we reconstruct the characteristic polynomial $\chi_G\ev{t}$ of a graph, from its \textit{polynomial deck}, i.e., the deck of characteristic polynomials of the vertex-deleted subgraphs? Some recent work on this and a discussion of the problem's history can be found in \cite{fa21,SciSta23}. Using the Harary-Sachs Theorem, Tutte showed in \cite[p.30]{tutte} that the characteristic polynomial of an ordinary graph is reconstructible from its vertex deck. Also, in \cite[p. 2]{SciSta23}, Sciriha and Stani\'c showed that there are non-isomorphic graphs with the same polynomial deck. %IMPROVED
In Figure \ref{fig1}, we present the table of four spectral reconstruction problems for ordinary graphs, from Schwenk's \textit{Spectral Reconstruction Problems} (\cite[p. 3]{schwenk}).

\begin{figure}[h]
\centering
\begin{tabularx}{96.55mm}{|c||c|c|}\hline 
\diagbox[width=4cm, height = 1.5cm]{Reconstruct}{Given}  & Deck of $G_i$ & Deck of $\chi_{G_i}\left(t\right)$'s  \tabularnewline[0.5cm]
\hline \hline \rule{0mm}{0.3cm} 
\raisebox{-0.2cm}{$G$} 	 & \tikzmark{a11}{ GRP (Open) }  & \tikzmark{a12}{ No (\cite{SciSta23}) } \tabularnewline[0.5cm]
\hline \rule{0mm}{0.5cm}
$\chi_G(t)$ & \tikzmark{a21}{ Yes (\cite{tutte})}   &   \tikzmark{a22}{ PRP (Open)  }  \tabularnewline[0.5cm]
\hline
\end{tabularx}
\link{a12}{a11}
\link{a12}{a22}
\link{a11}{a21}
\link{a22}{a21}

\caption{\label{fig1}Implications between reconstruction problems for ordinary graphs, where GRP stands for Graph Reconstruction Problem and PRP stands for Polynomial Reconstruction Problem.}
\end{figure}

For hypergraphs of rank $m\geq 3$, we have counterparts of the reconstruction conjectures. In particular, we have the polynomial reconstruction problem for hypergraphs.  Kocay showed in \cite{kocay} that hypergraphs of rank $3$ are not reconstructible from their decks. This is demonstrated by an infinite family of pairs of hypergraphs $\ev{X^n,Y^n}_{n\geq 3}$ such that their vertex-decks are the same, although $X^n$ and $Y^n$ are non-isomorphic. Our main theorem shows that the characteristic polynomials of these pairs of hypergraphs are different, established by showing that their principal eigenvalues are different. Since the polynomial decks of the pairs are the same, this disproves the polynomial reconstruction conjecture for hypergraphs of rank $3$. A novel aspect of our argument is that we treat hypergraphs throughout primarily as algebraic objects by analyzing their ``Lagrangians''. We suspect that similar methods can resolve the case of higher ranks as well, but have not attempted to work out the details.

\section{Preliminaries}
Throughout, we use the notation $N = 2^n$, where $n\geq 3$ in an integer. The set $\{1,2,\ldots,n\}$ is denoted as $[n]$, for any $n\geq 1$. We use $\mathbbm{1}_r$ to denote the all-ones vector of length $r\geq 1$. The subscript is omitted if it is clear from the context.  A {\em hypergraph} $\mathcal{H}$ is a pair $(V,E)$ of vertices $V(\mathcal{H}) := V$ and edges $E(\mathcal{H}) := E \subseteq 2^V$; if $|e|=m$ for all $e \in E$, then $\mathcal{H}$ is said to be {\em uniform of rank $m$}.

\begin{defn}
Given hypergraphs $\mathcal{H} = \ev{ [n], E\ev{ \mathcal{H} }}$ and $\mathcal{G} = \ev{ [n], E\ev{ \mathcal{G} }}$ of rank $m\geq 2$, then $\mathcal{H}$ and $\mathcal{G}$ are called \textit{hypomorphic}, provided there exists a permutation $\eta : [n] \rightarrow [n]$ so that $\mathcal{H} - i $ and $\mathcal{G} - \eta(i)$ are isomorphic, for each $i = 1,\ldots,n$.
\end{defn}
We have the following theorem by \cite{kocay}:
\begin{theorem}
There exists an infinite family of pairs of hypergraphs $\{\ev{X^n,Y^n}\}_{n\geq 3}$, of rank $3$, that are hypomorphic, but not isomorphic. 
\end{theorem}
We describe $X^n$ and $Y^n$ below in Definition \ref{def:final_defn}. Hypomorphic graphs share the same number of edges and the same degree sequences, among many other common properties. We will show that the characteristic polynomials of $X^n$ and $Y^n$ are different, for each $n \geq 3$.

\begin{defn}[Adjacency Hypermatrix]
Let $\mathcal{H} = \ev{ [n], E\ev{\mathcal{H}}}$ be a hypergraph, with $n$ vertices and rank $m\geq 2$. The (normalized) \textit{adjacency hypermatrix} $\mathcal{A}_\mathcal{H}$ of $\mathcal{H}$ is the rank $m$ and dimension $n$ symmetric hypermatrix with entries
$$ a_{i_1 \ldots i_m} = \begin{cases}  \ffrac{ 1 }{(m-1)!} & \text{ if } \{i_1,\ldots,i_m\} \in E\ev{\mathcal{H}} \\ 0 & \text{otherwise}
\end{cases}$$
\end{defn}

\begin{defn}
\label{def:lagrangian}
Given a hypergraph $\mathcal{H} = \ev{ [n], E\ev{\mathcal{H}}}$, with $n$ vertices and rank $m$, with adjacency hypermatrix $\mathcal{A}_\mathcal{H} = \ev{a_{i_1 \ldots i_m}}$, then we have the following multivariable function defined on $\R^n$, called the \textit{Lagrangian (polynomial)} of $\mathcal{H}$:
$$ F_\mathcal{H}\ev{\mathbf{x}} = \ev{1/m} \cdot \sum_{i_1 \ldots i_m = 1 }^{n} a_{i_1 \ldots i_m} x_{i_1} x_{i_2} \cdots x_{i_m}$$ %IMPROVED
For a given edge $e = \{i_1,\ldots, i_m\}\in E\ev{\mathcal{H}}$, let $\mathbf{x}^{e} = x_{i_1} x_{i_2} \cdots x_{i_m}$. With this notation, we have
$$ F_\mathcal{H}\ev{\mathbf{x}} = \sum_{e\in E\ev{ \mathcal{H} }} \textbf{x}^{e}$$ %IMPROVED
We will occasionally simplify the notation by writing $\mathcal{H}\ev{\mathbf{x}}$ for the Lagrangian $F_\mathcal{H}\ev{\mathbf{x}}$; in fact, we often suppress the vector of variables and use $\mathcal{H}$ to denote the Lagrangian as well as the hypergraph.
\end{defn}

\begin{defn}[Eigenpairs of tensors]
Given a hypergraph $\mathcal{H} = \ev{ [n], E\ev{\mathcal{H}}}$, of dimension $n$ and rank $m$, then a vector $\mathbf{x}\in \C^{n}$ is an \textit{eigenvector} of $\mathcal{H}$ corresponding to the \textit{eigenvalue} $\lambda \in \C$, provided the equation
$$ \sum_{i_2  \ldots i_m = 1 }^{n} a_{j i_2 \ldots i_m} x_{i_2} \cdots x_{i_m} = \lambda x_j^{m-1}$$
is satisfied, for each $j=1,\ldots,n$.
\end{defn}
Since the Lagrangian $F_\mathcal{H}$ is homogeneous by Definition \ref{def:lagrangian}, we typically consider only the restriction of $F_\mathcal{H}$ to the set $S$ of non-negative unit vectors (with respect to the $\ell_m$-norm):
$$S = \{ \ev{x_1,\ldots, x_n} \in \mathbb{R}_{\geq 0}^n : |x_1|^{m} +\ldots +|x_n|^{m} = 1 \}$$
Let $\textup{Int}\ev{S}$ be the relative interior of $S$, that is, the set of unit vectors with strictly positive coordinates. 
\begin{remark}
\label{rmk:past_theorems}
\phantom{a}

\begin{enumerate}
\item In \cite{cd1} (and, shortly after, in \cite{friedland} in greater generality), it was shown that, if $\mathcal{H}$ is connected, then there is a unique eigenvalue $\lambda$, associated with a unit eigenvector $\mathbf{v}\in \textup{Int}\ev{S}$ with strictly positive coordinates, analogously to the classical Perron-Frobenius Theorem. This eigenvalue, which is always real and strictly positive, is called the \textit{principal eigenvalue} of $\mathcal{H}$ or the \textit{spectral radius} of $\mathcal{H}$. The corresponding eigenvector $\mathbf{v}$ is the \textit{principal eigenvector} of $\mathcal{H}$.
\item As shown in \cite{cd1}, the function $F_\mathcal{H}\ev{\mathbf{x}}$ is maximized on $\text{Int}\ev{S}$ \textit{uniquely} at its principal eigenvector $\mathbf{v}$ and the maximum value attained is the principal eigenvalue $\lambda$ of $\mathcal{H}$. In other words,
$$ \max_{\mathbf{x} \in \textup{Int}\ev{S} } F_\mathcal{H}\ev{\mathbf{x}} =  F_\mathcal{H}\ev{\mathbf{v}} = \lambda $$ 
\end{enumerate}
\end{remark}
Next, we have a lemma about the relationship between the automorphisms of a hypergraph $\mathcal{H}$ and its principal eigenvector. 
\begin{defn}
Let $\sigma \in \text{Sym}\ev{[n]} $ be a permutation of $[n]$.
\begin{enumerate}
\item For a hypergraph $\mathcal{H} = \ev{ [n], E(\mathcal{H})}$, we write $\sigma(\mathcal{H})$ for the hypergraph whose vertices are $[n]$ and whose edges are $\{\sigma(e) : e \in E(\mathcal{H})\}$, where $\sigma(e) = \{\sigma(t) : t \in e\}$. 
\item We also apply the mapping $\sigma$ on the coordinates of a vector $\mathbf{x}$ via
$$ 
\sigma\ev{ x_1,\ldots,x_n} = \ev{x_{\sigma\ev{1}}, \ldots, x_{\sigma\ev{n}}  }
$$
\end{enumerate}

\end{defn}
Note that 
$$ F_{\sigma\ev{\mathcal{H}}}\ev{ \mathbf{x}} = F_\mathcal{H} \ev{ \sigma\ev{\mathbf{x}}}  $$
\begin{lemma}
\label{lem:automorphism_eigenvect}
Let $\mathcal{H} = \ev{[n],E\ev{ \mathcal{H} } }$ be a hypergraph, with principal eigenpair $\ev{ \lambda, \mathbf{v} }$. Then, 
$$ \sigma\ev{\mathbf{v}} = \mathbf{v} $$ 
for any automorphism $\sigma \in \textup{Aut}\ev{ \mathcal{H}}$ of $\mathcal{H}$.
\end{lemma}
\begin{proof}
For each $\sigma \in \textup{Aut}\ev{ \mathcal{H}}$, we have
$$ F_\mathcal{H} \ev{ \mathbf{x} } = F_{\sigma\ev{\mathcal{H}}}\ev{ \mathbf{x}} $$
Evaluating at the principal eigenvector, we obtain,
$$ 
\lambda  = F_\mathcal{H} \ev{ \mathbf{v} } = F_{\sigma\ev{\mathcal{H}}}\ev{ \mathbf{v}} = F_\mathcal{H} \ev{ \sigma\ev{\mathbf{v}}} 
$$
By the uniqueness of the maximum point of $F_\mathcal{H} $ (Remark \ref{rmk:past_theorems}, part 2), it follows that $ \sigma\ev{\mathbf{v}} = \mathbf{v}$.
\end{proof}

Now, we consider the infinite family of pairs of hypergraphs $\{\ev{X^n,Y^n}\}_{n\geq 3}$ constructed in \cite{kocay}.

\begin{defn}
For $n\geq 0$, let $V_n = \{1,2,3,4,\ldots, 2^n\}$. Given $x,y\in V_n$, we define 
$$x+y \pmod{V_n}  = z$$
where $z$ is the unique integer $z\in V_n$ such that $x+ y \equiv z \pmod{2^n}$. Throughout, the superscript $n$ in $\mathcal{H}^{n}$ for each hypergraph $\mathcal{H}$ will indicate that calculations are conducted $\Mod{V_n}$.
\end{defn}

%p_{\mathrm{odd}},p_{\mathrm{even}}
\begin{defn}
For each $n\geq 4$, we define the maps $p_0^{n},p_1^{n} : V_{n-1}\rightarrow V_n$, by 
\begin{align*}
&p_0^{n} \ev{ i }  = 2i \pmod {V_n}  \\
&p_1^{n}\ev{  i }  = 2i-1 \pmod {V_n}
\end{align*}
for $i\geq 1$. 

We use the same notation for the polynomial map $p_j^{n}\ev{ x_i } = x_{ p_j^{n} \ev{ i }} $, for $j=0,1$. The maps $p_0^{n},p_1^{n}$ extend to a $\C$-algebra endomorphism of $\C[x_0,x_1,\ldots,x_N]$. We omit the superscript if it is clear from context.

For $r\geq 1$ and for a given zero-one vector $\bm{\epsilon} = [\epsilon_0,\ldots,\epsilon_{r-1}] \in \{0,1\}^{r}$, we let 
$$
p_{\bm{\epsilon}} = p_{\epsilon_0} \circ \ldots \circ p_{\epsilon_{r-1} }
$$
\end{defn}

\begin{remark}
\label{rmk:repeated_app}
It follows, by induction, that 
$$ p_{\bm{\epsilon}}\ev{ i } = 2^{r} i - \sum_{j = 0}^{ r-1 } \epsilon_j 2^{j} \qquad \text{ for each $\bm{\epsilon}\in \{0,1\}^{r}$ }$$ 
\end{remark}

\begin{defn}
Define the permutation $\tau \in \textup{Sym}([8])$ by 
$$ \tau \ev{ i }  = 3i \pmod { V_3 }$$
In other words, $\tau 
=\scalebox{1.4}{$\binom{12345678}{36147258}$}$.
\end{defn}

\begin{defn}
\label{def:two_cycles}
Define the tight cycle,
$$
C^{3}\ev{ \mathbf{x} } = \sum_{1\leq i \leq 8} x_{i-1} x_{i} x_{i+1}
$$
We use $\tau$ to obtain another tight cycle: 
$$
D^{3} \ev{ \mathbf{x} } = C^{3} \ev{ \tau\ev{ \mathbf{x} } } 
$$
\end{defn} 
\begin{remark}
Note that
\begin{align*}
& D^{3} \ev{ \mathbf{x} } = \sum_{1\leq i \leq 8} x_{3i-3} x_{3i} x_{3i+3} \\ 
& = \sum_{1\leq j \leq 8} x_{j-3} x_{j} x_{j+3}  & \text{ by change of variables, $j = 3i \Mod{V_3}$, $1\leq i \leq 8$ } 
\end{align*}
\end{remark}

\begin{defn}
For $n\geq 3$ and $2\leq r\leq n$, we define a map $\textup{E}^{n}_r$:
\begin{enumerate}
\item For $r=n$, we define the identity map:
$$\textup{E}^{n}_n \ev{ x_i } = x_i $$
\item For $2 \leq r \leq n-1$, we define
$$\textup{E}^{n}_r  \ev{ x_i }  = \sum_{ 1 \leq s \leq 2^{n-r} } x_{ i + s \cdot 2^r  } = \sum_{\substack{ 1\leq j \leq 2^n  \\ j \equiv i \Mod{2^r} } } x_j $$
\end{enumerate}
The map $\textup{E}^{n}_r$ extends to a $\C$-algebra endomorphism of $\C[x_0,x_1,\ldots,x_N]$. We omit the superscript if it is clear from the context. 
\end{defn}

\begin{remark}
\label{rmk:sigma_exceeds_threshold}
\phantom{a} 
$ \textup{E}^{n}_r  \ev{ x_i } = \text{E}^{n}_r  \ev{ x_{i+2^t} }$ for any $t\geq r$.
\end{remark}

\begin{defn}
\label{def:cycles}
Let $k\geq 2$ and $n\geq 3$. We define a family of hypergraphs $G_{k}^n$ as follows:
\begin{enumerate} 
\item[(1)] If $k=n$, we define :

For $n=3$, 
$$
G_{3}^3 \ev{ \mathbf{x} } = C^{3} \ev{ \mathbf{x} } + D^{3} \ev{ \mathbf{x} } =  \sum_{1\leq i \leq 8} x_{i-1} x_{i} x_{i+1} + \sum_{1\leq i \leq 8} x_{i-3} x_{i} x_{i+3}
$$

For $n\geq 4$, 
$$
G_{n}^n \ev{\mathbf{x}} = \textup{E}_3^{n} \ev{ G_3^{3} \ev{ \mathbf{x} } }
$$

\item[(2)] If $2\leq k <n$, we define:

For $n=3$, 
$$
G_{2}^3 \ev{ \mathbf{x} } = \sum_{1\leq i \leq 4} x_{i} x_{i+2} x_{i+4} 
$$

For $n\geq 4$, 
$$
G_{k}^n\ev{ \mathbf{x} } = p_{0}\ev{ G_{k}^{n-1} \ev{ \mathbf{x} }} + p_{1}\ev{ G_{k}^{n-1} \ev{ \mathbf{x} } } 
$$
\end{enumerate}
\end{defn}
The hypergraphs $G_{1,k}^{0}\ev{n}$ defined in \cite{kocay} are shortly denoted here as $G_k^{n}$.

%\begin{remark}
%Note that 
%\begin{align*}
%& V\ev{G_n\ev{n}} = V_n \text{ and } \lnm G_3\ev{3} \rnm = 16 \\
%& E\ev{ G_{3}\ev{3} } \equiv \left\{\{0,1,2\},\{1,2,3\},\{2,3,0\},\{3,0,1\}\right\} \pmod{4}\\
%\end{align*}
%As the function $\text{Eight}$ preserves congruence classes modulo 4, it follows by induction that $E\ev{G_{n}\ev{n}}  \equiv \left\{\{0,1,2\},\{1,2,3\},\{2,3,0\},\{3,0,1\},\{0,1,2\}\right\} \pmod{4}$ for $n\geq 3$. 
%\end{remark}
%

\begin{defn}
\label{def:infinity_vertex_defn}
\begin{align*}
& M^n_0 \ev{ \mathbf{x} }= x_0 \textup{E}^{n}_2 \ev{  x_1 x_2 + x_3 x_4} \\
& M^n_1 \ev{ \mathbf{x} }=x_0 \textup{E}^{n}_2 \ev{  x_1 x_4 + x_2 x_3 }  
\end{align*}
The hypergraphs defined as $M^{0}_n\ev{ n }, M^{1}_n\ev{ n }$ in \cite{kocay} are denoted here as $M^n_0$ and $M^{n}_1$, respectively.
\end{defn}

\begin{defn}
\label{def:final_defn}
Let $n\geq 3$. 
Define 
\begin{enumerate}
\item[i)] $$T_n\ev{\mathbf{x}} = \sum_{k=2}^{n-1} G_{k}^{n}\ev{\mathbf{x}} $$
\item[ii)] $$\Gamma_n\ev{\mathbf{x}} = T_n\ev{\mathbf{x}} + G_n^{n}\ev{\mathbf{x}} = \sum_{k=2}^{n} G_{k}^{n}\ev{\mathbf{x}} $$
\item[iii)]
$$X^n \ev{\mathbf{x}}= \Gamma_n\ev{\mathbf{x}} + M^n_0 \ev{\mathbf{x}}$$
$$Y^n \ev{\mathbf{x}} = \Gamma_n \ev{\mathbf{x}} + M^n_1 \ev{\mathbf{x}} $$
\end{enumerate}
The hypergraphs $G_n$, $X_n$, $Y_n$ defined in \cite{kocay} are denoted here as $\Gamma_n$, $X^n$, $Y^n$, respectively. As shown in \cite{kocay}, the hypergraphs $X^n$ and $Y^n$, for $n\geq 3$, are hypomorphic, but not isomorphic. 

\end{defn}
Note that, by Definition \ref{def:infinity_vertex_defn}, the vertex $0$ has positive codegree with every other vertex (i.e., they are contained in an edge together), so $X^n$ and $Y^n$ are connected, and therefore the hypergraph Perron-Frobenius Theorem applies as per Remark \ref{rmk:past_theorems}.

The hypergraph $G_n^{n}$ was defined in Definition \ref{def:cycles} with an explicit formula. In \cite{kocay}, a recursive definition of $G_{1,n}^{0}\ev{n}$ is given. To show that these definitions are equivalent, we define below $H^{n}$, which is an alternative notation for $G_{1,n}^{0}\ev{n}$.

\begin{defn}
For each $n\geq 3$, we define the mapping %IMPROVED
$$ q^{n}\ev{ x_i } = x_i + x_{i+2^{n-1}} $$
We extend the mapping $q^{n}$ to a $\C$-algebra endomorphism of $\C[x_0,x_1,\ldots,x_N]$.
%We omit the superscripts, whenever it is clear from the context. 
\end{defn}
We note that $q^{n} = \textup{E}_{n-1}^{n}$, for each $n\geq 3$. %IMPROVED
\begin{defn}
Let $n\geq 3$. We define the family $H^n$ as follows:

For $n=3$, 
$$
H^3 \ev{ \mathbf{x} } = G_{3}^3 \ev{ \mathbf{x} } = C^{3} \ev{ \mathbf{x} } + D^{3} \ev{ \mathbf{x} } 
$$

For $n\geq 4$, 
$$
H^n \ev{\mathbf{x}} = q^{n} \ev{ H^{n-1}\ev{ \mathbf{x} } }
$$
\end{defn}

\begin{lemma}
\label{lem:rec_defn_g_n}
We claim that 
$$
H^{n}  = \textup{E}^{n}_{3} \ev{ G_3^{3} } = G_{n}^{n}
$$
for each $n\geq 3$. In other words, the hypergraph $G_n^{n}$ defined in Definition \ref{def:cycles} is identical to that defined in \cite{kocay}.
\end{lemma}
\begin{proof}
First, we show that 
$$ q^{n} \textup{E}^{n-1}_{3}  = \textup{E}^{n}_{3}$$
as follows: 
\begin{align*}
& \textup{E}^{n}_{3} \ev{ x_i } = \sum_{ 1\leq s \leq 2^{n-3} } x_{ i + 8s } = \sum_{ 1\leq t \leq 2^{n-4} } x_{ i + 8t } + x_{ i + 8t + 2^{n-1} } \\  
& = \sum_{ 1\leq t \leq 2^{n-4} } q^{n} \ew{ x_{ i + 8t } } = q^{n} \ew{ \sum_{ 1\leq t \leq 2^{n-4} } x_{ i + 8t } } = q^{n} \ev{ \textup{E}^{n-1}_{3} \ev{ x_i }} 
\end{align*}
As the maps agree on the generators of $\C[x_0,x_1,\ldots,x_N]$, they are identical.

For the statement of the lemma, we use induction on $n$, to prove that 
$$  
H^{n} = \textup{E}^{n}_{3} \ev{ G_3^{3} }.
$$
The base case of $n=3$ is clear. For the induction step, 
\begin{align*}
& H^{n} = q^{n}\ev{ H^{n-1} } \\ %IMPROVED
& = q^{n} \textup{E}^{n-1}_{3} \ev{ G_3^{3} } && \text{ by the inductive hypothesis} \\ %IMPROVED
& = \textup{E}^{n}_{3} \ev{ G_3^{3} } && \text{ as shown above }
\end{align*}
\end{proof}

The recursive definitions of $\{ G_{k}^{n} \}_{k=2}^{n}$, $T_n$ and $\Gamma_n$, found in Definitions \ref{def:cycles} and \ref{def:final_defn} can be turned into explicit formulas, by induction:
\begin{remark}
\phantom{a}

\label{rmk:rec_defn}
\begin{enumerate}
\item $$G_{2}^{n} = \sum_{ \bm{\epsilon} \in \{0,1\}^{n-3} } p_{\bm{\epsilon}} \ev{ G_2^{3} } = \sum_{1\leq i \leq 4} \sum_{ \bm{\epsilon} \in \{0,1\}^{n-3} } p_{\bm{\epsilon}} \ev{  x_{i} x_{i+2} x_{i+4} } $$ 
\item For $3\leq k \leq n$, $$G_{k}^{n} = \sum_{ \bm{\epsilon} \in \{0,1\}^{n-k} } p_{\bm{\epsilon}} \ev{ G_k^{k} } \text{ where } G_k^{k}  = \textup{E}^{k}_3 \ev{ G_3^{3} }$$
\item $ \Gamma_n = T_n + G_n^{n}$, where
$$T_n = \sum_{1\leq i \leq 4} \sum_{ \bm{\epsilon} \in \{0,1\}^{n-3} } p_{\bm{\epsilon}} \ev{ x_{i} x_{i+2} x_{i+4} } + \sum_{1\leq i \leq 8} \sum_{1 \leq r\leq n-3 } \sum_{ \bm{\epsilon} \in \{0,1\}^{r} } p_{\bm{\epsilon}} \textup{E}_{n-r} \ev{ x_{i-1} x_{i} x_{i+1} + x_{i-3} x_{i} x_{i+3} }  $$
and 
$$ G_n^n = \textup{E}^{n}_3 \ev{ G_3^{3} } =  \textup{E}^{n}_{3} \lb \sum_{1\leq i \leq 8}  x_{i-1} x_{i} x_{i+1} + x_{i-3} x_{i} x_{i+3}  \rb $$
\end{enumerate}
\end{remark}
Note that each edge of $T_n$ consists of vertices of the same parity. 

\section{Main Theorem}
By \cite[Theorems 4.7 and 5.4]{kocay}, we know that $\text{Aut}\ev{X^n} = \text{Aut}\ev{Y^n} = \{ \textup{id}, \theta_n \}$, where $\theta_n$ fixes $0$ and reflects the cycle $1,\ldots,2^n$:
$$
\theta_n\ev{x} = 
\begin{cases} 2^n - x + 1 &\text{ if } 1\leq x \leq 2^n \\ 0 & \text{ if } x=0 \end{cases}
$$
Note that $\theta_n$ is an involution and $|\text{Aut}\ev{X^n}| = |\text{Aut}\ev{Y^n}| =2$. 

\begin{lemma}
\label{lem:basis_step}
Let $\mathbf{x} \in \C^{N+1}$ be a vector such that $\theta_n\ev{\mathbf{x}} = \mathbf{x}$. Then, we have
$$ Y^n \ev{ \mathbf{x} } - X^n \ev{ \mathbf{x} }  = x_0 \lb \textup{E}_2 \ev{ x_1 - x_3 } \rb^{2}  $$
\end{lemma}
\begin{proof}
%The common automorphism $\theta_n$ of $X^n$ and $Y^n$ is the mapping 
%
%$$ \theta_n: [n] \cup \{0\} \rightarrow [n] \cup \{0\}$$
%given by $0 \mapsto 0, \text{ and }  x \mapsto 2^n - x + 1$ for $ x \neq 0$.

We calculate,
\begin{align*}
& Y^n - X^n =  ( \Gamma_n  + M^n_{1} ) - ( \Gamma_n  + M^n_0  )  \\
& =  M^n_1 - M^n_0  = x_0 \textup{E}_2 \ev{ x_1 x_4 + x_2 x_3 } - x_0 \textup{E}_2 \ev{ x_1 x_2 + x_3 x_4 } && \text{ by Definition \ref{def:infinity_vertex_defn} } \\ 
& = x_0 \textup{E}_2 \ev{ x_1 x_4 + x_2 x_3 - x_1 x_2 - x_3 x_4} \quad  && \text{ since $\textup{E}_2$ is a homomorphism } \\ 
& = x_0 \textup{E}_2  \ev{ x_1 - x_3 } \textup{E}_2\ev{ - x_2 + x_4 }  \quad  &&  \\ 
& = x_0  \textup{E}_2  \ev{ x_1 - x_3 } \textup{E}_2\ev{ - x_{2^n - 1} + x_{2^n-3} } &&
\end{align*}
because $\mathbf{x} = \theta_n(\mathbf{x})$ implies $x_2 = x_{2^n-1}$ and $x_4 = x_{2^n-3}$.  Continuing, we may write
\begin{align*}
& Y^n - X^n = x_0 \textup{E}_2 \ev{ x_1 - x_3 } \textup{E}_2 \ev{ - x_{4 - 1} + x_{4-3} }   \quad  && \text{ by Remark \ref{rmk:sigma_exceeds_threshold} } \\
& = x_0 \lb \textup{E}_2 \ev{ x_1 - x_3 } \rb^{2}.   \quad  &&   
\end{align*}
\end{proof}

Let $\ev{ \lambda, \mathbf{x} }, \ev{ \mu, \mathbf{y}}$ be the principal eigenpairs of $X^n$ and $Y^n$, respectively. As $\theta_n$ is an automorphism of $X^n$, it follows, by Lemma \ref{lem:automorphism_eigenvect}, that $\theta_n\ev{\mathbf{x}} = \mathbf{x}$. In particular, using Lemma \ref{lem:basis_step}, we obtain
$$ \mu = Y^n\ev{ \mathbf{y} } \geq Y^n\ev{ \mathbf{x} } = X^n \ev{ \mathbf{x} } +  x_0 \lb \textup{E}_2 \ev{ x_1 - x_3 } \rb^{2} = \lambda +  x_0 \lb \textup{E}_2 \ev{ x_1 - x_3 } \rb^{2}$$
In other words, the principal eigenvalue of $Y^n$ is greater than or equal to that of $X^n$. We claim that they are not equal. For this purpose, it is enough to show that $\textup{E}_2 \ev{ x_1 - x_3 } $ is non-zero. This is proven in Lemma \ref{lem:main_lemma} below. 

\begin{defn}
Given $i \geq 0$, we define a mapping $\sigma_i: \Z^+ \rightarrow \Z^+$, as follows:
$$ 
\sigma_i\ev{ j } = \begin{cases} j + 2^{i} & \text{ if } j\equiv 1,\ldots,2^i \pmod{2^{i+1}} \\ j - 2^{i} & \text{ if } j\equiv 1+2^i,\ldots,2^{i+1} \pmod{2^{i+1}}
\end{cases}
$$
We put $\sigma_{-1} = \mathrm{id}$. We also allow $\sigma_i$ to act on $\C[x_0,x_1,\ldots,x_N]$, via $\sigma_i\ev{ x_j } = x_{ \sigma_i\ev{ j }}$. 
\end{defn}

For example, the restriction of $\sigma_0$ to $[8]$ is the permutation \scalebox{1.4}{$\binom{12345678}{21436587}$}, whereas the restriction of $\sigma_1$ to $[8]$ is the permutation \scalebox{1.4}{$\binom{12345678}{34127856}$}.

\begin{lemma}
\label{lem:helper_cycle}
Let $\mathbf{x}\in \C^{N+1}$ such that $\theta_n\ev{\mathbf{x}} = \mathbf{x}$. 
\begin{enumerate}
\item $\textup{E}_3\ev{x_{2i}} = \textup{E}_3\ev{x_{9-2i}}$ for each $i=1,2,3,4$.
\item If $\mathbf{x} = \sigma_0\ev{ \mathbf{x} }$, then we have
$$  \textup{E}_3\ev{ x_7 } =  \textup{E}_3\ev{ x_1 } \text{ and }  \textup{E}_3\ev{ x_5 } =  \textup{E}_3\ev{ x_3 }  $$
\item If $\mathbf{x} = \sigma_0\ev{ \mathbf{x} } = \sigma_1\ev{ \mathbf{x} }$, then we have
$$  \textup{E}_3\ev{ x_1 } =  \textup{E}_3\ev{ x_2 } = \ldots = \textup{E}_3\ev{ x_8 }  $$
\end{enumerate}
\end{lemma}
\begin{proof}
\begin{enumerate}
\item We apply $\theta_n$ and Remark \ref{rmk:sigma_exceeds_threshold}, to obtain, for each $i=1,2,3,4$, 
$$ 
\textup{E}_3\ev{x_{2i}} = \textup{E}_3\ev{x_{2^n - 2i + 1}} = \textup{E}_3\ev{x_{8 - 2i + 1}} = \textup{E}_3\ev{x_{9- 2i}}
$$
\item We have,
\begin{align*}
& \textup{E}_3\ev{ x_7 }  = \textup{E}_3\ev{ x_2 }	&& \text{ by Part 1} \\ 
& = \textup{E}_3\ev{ x_1 } 										&&\text{ since $\sigma_0\ev{ \mathbf{x} } = \mathbf{x}$ } 
\end{align*}
By a similar calculation, we have $\textup{E}_3\ev{ x_3 } = \textup{E}_3\ev{ x_5 }$.

\item By the assumption that $\sigma_1\ev{ \mathbf{x} } = \mathbf{x}$, we obtain
$$ \textup{E}_3\ev{ x_{5} } = \textup{E}_3\ev{ x_{1} } $$
Combined with Part 1 and Part 2, the proof is complete.  

%
%\item Since $\sigma_0\ev{ \mathbf{x} } = \mathbf{x}$, we have 
%$$\textup{E}_3\ev{ x_{2i} } = \textup{E}_3\ev{ x_{2i-1} }$$
%for each $i=1,2,3,4$. 
%Also, by Part 1 and the assumption that $\sigma_1\ev{ \mathbf{x} } = \mathbf{x}$, we get
%$$ 
%\textup{E}_3\ev{ x_3 } = \textup{E}_3\ev{ x_{5} } = \textup{E}_3\ev{ x_{1} } = \textup{E}_3\ev{ x_7 } 
%$$

\end{enumerate}
\end{proof}

\begin{defn}
Given $0\leq r \leq n-2$ and $\mathbf{x} = \ev{x_1,\ldots,x_N} $, we define 
$$
f^{n}_r\ev{\mathbf{x}} = 
\begin{cases}
2^{r+1} \textup{E}^{n}_{r+2} \ev{  x_1 - x_{ 1 + 2^{r+1} } } \cdot \textup{E}^{n}_{r+3} \ev{ x_{1} - x_{ 1 +  2^{r+2} } }   \cdot  \textup{E}^{n}_{r+3} \ev{  - x_{ 1 + 2^{r+1}  } + x_{ 1 + 2^{r+1} + 2^{r+2} } } \\ 
\hfill \text{ if } 0\leq r \leq n-3 \\ 
2^{n-1} \ev{  x_1 - x_{ 1 + 2^{n-1} } } x_{1}   \ev{ - x_{ 1 + 2^{n-1}  } } \\
\hfill \text{ if } r = n-2
\end{cases} 
$$
\end{defn}
Recall that $\textup{E}^{n}_{n}$ acts as the identity. Hence, for each $0\leq r\leq n-2$, the polynomial $f^{n}_r\ev{\mathbf{x}}$ is divisible by $\textup{E}^{n}_{r+2} \ev{  x_1 - x_{ 1 + 2^{r+1} } }$. 

Now, we develop some properties of the polynomials $\{ f^{n}_r\ev{\mathbf{x}} \}_{0 \leq r \leq n-2}$. 
\begin{lemma}
\label{lem:odd_even_eight}
Given $\mathbf{x} = \ev{x_0,x_1,\ldots,x_N} $ such that $\theta_n \ev{\mathbf{x}} = \mathbf{x}$ and $\sigma_0\ev{ \mathbf{x} } = \mathbf{x}$, then for any $2\leq r \leq n-2$, and $j = 0,1$:
\begin{align*}
p_j   \textup{E}_{r}^{n} \ev{ x_{1} } 		&= \textup{E}^{n+1}_{r+1}\ev{ x_{1}}	 \\
p_j  \textup{E}_{r}^{n} \ev{ x_{1+2^{r-1}} } &= \textup{E}^{n+1}_{r+1}\ev{ x_{1+2^r}} 
\end{align*}
\end{lemma}
\begin{proof}
Note that $p_0 \textup{E}_{r}^{n} \ev{ x_{i} } = \textup{E}^{n+1}_{r+1}\ev{ x_{2i} }$ and $p_1 \textup{E}_{r}^{n} \ev{ x_{i} } = \textup{E}^{n+1}_{r+1}\ev{ x_{2i-1} }$. In particular, we obtain 
\begin{align*}
p_1 \textup{E}_{r}^{n} \ev{ x_{1} } & =  \textup{E}^{n+1}_{r+1}\ev{ x_{1}}  \\ 
p_1 \textup{E}_{r}^{n} \ev{ x_{1+2^{r-1}} } & = \textup{E}^{n+1}_{r+1}\ev{ x_{1+2^r}} 
\end{align*}
and also, using the assumption that $\sigma_0\ev{ \mathbf{x} } = \mathbf{x}$, we obtain
\begin{align*}
&p_0 \textup{E}_{r}^{n} \ev{ x_{1} } = \textup{E}^{n+1}_{r+1}\ev{ x_{2}} = \textup{E}^{n+1}_{r+1}\ev{ x_{1}} \\
&p_0 \textup{E}_{r}^{n} \ev{ x_{1+2^{r-1}} } = \textup{E}^{n+1}_{r+1}\ev{ x_{2+2^{r} }}  = \textup{E}^{n+1}_{r+1}\ev{ x_{1+2^{r+1} }} 
\end{align*}
\end{proof}

\begin{lemma}
\label{lem:even_odd_f}
If $\mathbf{x} = \ev{x_0,x_1,\ldots,x_N} $ is such that $\theta_n \ev{\mathbf{x}} = \mathbf{x}$ and $\sigma_0\ev{ \mathbf{x} } = \mathbf{x}$, then 
$$\ev{p_0 + p_1}\ev{ f^{n}_r\ev{ \mathbf{x} } } = f_{r+1}^{n+1}\ev{ \mathbf{x} }$$
\end{lemma}
\begin{proof} As $p_0$ and $p_1$ are homomorphisms, we obtain, by Lemma \ref{lem:odd_even_eight}, that

\textbf{Case 1:} $0 \leq r \leq n-3$.
\begin{align*}
& p_0 \ew{ f^{n}_r\ev{ \mathbf{x} } } + p_1 \ew{ f^{n}_r\ev{ \mathbf{x} } } \\
& = p_0 \lb 2^{r+1} \textup{E}^{n}_{r+2} \ev{  x_1 - x_{ 1 + 2^{r+1} } } \cdot \textup{E}^{n}_{r+3} \ev{ x_{1} - x_{ 1 +  2^{r+2} } }   \cdot  \textup{E}^{n}_{r+3} \ev{  - x_{ 1 + 2^{r+1}  } + x_{ 1 + 2^{r+1} + 2^{r+2} } } \rb \\
& + p_1 \lb 2^{r+1} \textup{E}^{n}_{r+2} \ev{  x_1 - x_{ 1 + 2^{r+1} } } \cdot \textup{E}^{n}_{r+3} \ev{ x_{1} - x_{ 1 +  2^{r+2} } }   \cdot  \textup{E}^{n}_{r+3} \ev{  - x_{ 1 + 2^{r+1}  } + x_{ 1 + 2^{r+1} + 2^{r+2} } } \rb \\
& = 2 \cdot 2^{r+1} \textup{E}^{n+1}_{r+3} \ev{  x_1 - x_{ 1 + 2^{r+2} } } \cdot \textup{E}^{n+1}_{r+4} \ev{ x_{1} - x_{ 1 +  2^{r+3} } }   \cdot  \textup{E}^{n+1}_{r+4} \ev{  - x_{ 1 + 2^{r+2}  } + x_{ 1 + 2^{r+2} + 2^{r+3} } } \\
& = f_{r+1}^{n+1}\ev{ \mathbf{x} }.
\end{align*}

\textbf{Case 2:} $r=n-2$. 
\begin{align*}
& p_0 \ew{ f^{n}_{n-2}\ev{ \mathbf{x} } }+ p_1\ew{ f^{n}_{n-2}\ev{ \mathbf{x} } } \\
& = p_0 \lb 2^{n-1} \ev{  x_1 - x_{ 1 + 2^{n-1} } } x_{1}  \ev{ - x_{ 1 + 2^{n-1}  }   } \rb + p_1 \lb 2^{n-1} \ev{  x_1 - x_{ 1 + 2^{n-1} } } x_{1}  \ev{ - x_{ 1 + 2^{n-1}  }   } \rb  \\
& = 2 \cdot 2^{n-1} \ev{  x_1 - x_{ 1 + 2^{n} } } x_{1}  \ev{ - x_{ 1 + 2^{n}  }   } \\
& = f_{n-1}^{n+1}\ev{ \mathbf{x} }
\end{align*}
\end{proof}

\begin{remark}
\label{rmk:interaction_sigma_p}
It is easy to check that, for each $i \in V_{n-1}$, $1\leq r \leq n-2$ and $j\in \{0,1\}$, we have
$$
p_j \sigma_{r-1}\ev{ x_i } = \sigma_r p_j\ev{ x_i } 
$$
%p_1 \sigma_{r-1}\ev{ x_i } = \sigma_r p_1\ev{ x_i } 

\end{remark}

%The following lemma is concerned with the effect of applying the permutation $\sigma_r$ on the coordinates of 
\begin{lemma}
\label{lem:induction_cycles}
Let $n\geq 3$, $0\leq r \leq n-2$ and $2\leq k \leq n$ be fixed. Let $\mathbf{x}\in \C^{N+1}$ such that $\theta_n\ev{\mathbf{x}} = \mathbf{x}$ and $\mathbf{x} = \sigma_i\ev{ \mathbf{x} }$ for each $i=-1,0,\ldots, r-1$. Then,
$$ G_k^{n}\ev{ \mathbf{x} } - G_k^{n}\ev{ \sigma_{r} \ev{  \mathbf{x}  }} = \begin{cases}  f_{r}^{n} \ev{ \mathbf{x} } & \text{ if } k = n-r \\ 0 & \text{ if }  k\neq n-r
\end{cases} $$
\end{lemma}
\begin{proof}
\textbf{Case 1:} $k=n$. In this case, we would like to show that 
$$ 
G_n^{n}\ev{ \mathbf{x} } - G_n^{n}\ev{ \sigma_{r} \ev{  \mathbf{x}  }} = 
\begin{cases}  
f_{0}^{n} \ev{ \mathbf{x} } & \text{ if } r=0 \\ 
0 & \text{ if }  1\leq r 
\end{cases} 
$$
By Definition \ref{def:cycles}, have
\begin{equation}
\label{eq:G_n_defn}
G_n^n = \textup{E}_3 \ev{ G_3^3 }
\end{equation}
By Remark \ref{rmk:sigma_exceeds_threshold}, we have $\text{E}^{n}_3  \ev{ x_i } =  \text{E}^{n}_3  \ev{ \sigma_t\ev{ x_i } }$
whenever $t\geq 3$. Therefore, we only need to consider $r = 0,1,2$.

\begin{enumerate}
\item[a)] Assume $r=0$. By Lemma \ref{lem:helper_cycle}, we have 
\begin{align}
\begin{split}
\label{eq:dct_subs_theta}
\textup{E}_3\ev{ x_2 } = \textup{E}_3\ev{ x_7 } \qquad \text{and} \qquad \textup{E}_3\ev{ x_4 }= \textup{E}_3\ev{ x_5 } \\ 
\textup{E}_3\ev{ x_6 } = \textup{E}_3\ev{ x_3 } \qquad \text{and} \qquad 
\textup{E}_3\ev{ x_8 }= \textup{E}_3\ev{ x_1 }
\end{split}
\end{align}
We calculate,
$$
\begin{tblr}{}
C^{3}\ev{ \mathbf{x} } - C^{3} \ev{ \sigma_{0} \ev{  \mathbf{x}  }} \\
= ( x_1 x_2 x_3 + x_2 x_3 x_4 + x_3 x_4 x_5 + x_4 x_5 x_6 + x_5 x_6 x_7 + x_6 x_7 x_8 + x_7 x_8 x_1 + x_8 x_1 x_2)\\
- \ev{ x_2 x_1 x_4 + x_1 x_4 x_3 + x_4 x_3 x_6 + x_3 x_6 x_5 + x_6 x_5 x_8 + x_5 x_8 x_7 + x_8 x_7 x_2 + x_7 x_2 x_1 } \\
& \hspace{-5cm} \text{by Definition \ref{def:two_cycles}}\\	
=  2 x_1 x_3 x_7 + 2 x_3 x_5 x_7 + 2 x_3 x_5^2 + 2 x_1^2 x_7 - 2 x_1 x_5 x_7  - 2 x_1 x_3 x_5 - 2 x_3^2 x_5 - 2 x_1 x_7^2  \\
& \hspace{-5cm} \text{by Equations \ref{eq:dct_subs_theta}}
\end{tblr}
$$
On the other hand,
$$
\begin{tblr}{}
D^{3}\ev{ \mathbf{x} } - D^{3}\ev{ \sigma_{0} \ev{  \mathbf{x}  }} \\
= \ev{ x_1 x_4 x_6 + x_2 x_5 x_7 + x_3 x_6 x_8 + x_4 x_7 x_1 + x_5 x_8 x_2 + x_6 x_1 x_3 + x_7 x_2 x_4 + x_8 x_3 x_5}\\	
-\ev{ x_2 x_3 x_5 + x_1 x_6 x_8 + x_4 x_5 x_7 + x_3 x_8 x_2 + x_6 x_7 x_1 + x_5 x_2 x_4 + x_8 x_1 x_3 + x_7 x_4 x_6}\\
& \hspace{-5cm} \text{by Definition \ref{def:two_cycles}}\\		
= 2 x_1 x_3 x_5  + 2 x_5 x_7^2 + 2 x_1 x_3^2 + 2 x_1 x_5 x_7    -  2 x_3 x_5 x_7  - 2 x_1^2 x_3  - 2 x_5^2 x_7 - 2 x_1 x_3 x_7   \\	
& \hspace{-5cm} \text{by Equations \ref{eq:dct_subs_theta}}
\end{tblr}
$$
Therefore, 
$$
\begin{tblr}{}
& G_n^{n}\ev{ \mathbf{x} } - G_n^{n}\ev{ \sigma_{0} \ev{  \mathbf{x}  } }  = E_3 \ev{ C^{3} \ev{ \mathbf{x} } + D^{3} \ev{ \mathbf{x} } } -  E_3 \ev{ C^{3} \ev{ \sigma_{1} \ev{  \mathbf{x}  } } + D^{3} \ev{ \sigma_{1} \ev{  \mathbf{x}  } } } \\
& = \textup{E}_3 \ev{ C^{3} \ev{ \mathbf{x} } - C^{3} \ev{ \sigma_{1} \ev{  \mathbf{x}  }  } } + E_3 \ev{ D^{3} \ev{ \mathbf{x} } - D^{3} \ev{ \sigma_{1} \ev{  \mathbf{x}  } } }  \\ 
& = \textup{E}_3 (   2 x_1 x_3 x_7 + 2 x_3 x_5 x_7 + 2 x_3 x_5^2 + 2 x_1^2 x_7 -  2 x_1 x_5 x_7  - 2 x_1 x_3 x_5  - 2 x_3^2 x_5 - 2 x_1 x_7^2  ) \\
& + \textup{E}_3 (  2 x_1 x_3 x_7 + 2 x_3 x_5 x_7 + 2 x_3 x_5^2 + 2 x_1^2 x_7 -  2 x_1 x_5 x_7  - 2 x_1 x_3 x_5  - 2 x_3^2 x_5 - 2 x_1 x_7^2 )   \\ 
& = 2 \textup{E}_3 \ev{ x_3 x_5^2   +  x_1^2 x_7  +   x_5 x_7^2 +   x_1 x_3^2  - x_3^2 x_5  - x_1 x_7^2 - x_5^2 x_7 - x_1^2 x_3 }\\
& = 2\textup{E}_3 \lb  (x_1 - x_3 + x_5 - x_7) \ev{ x_1 - x_5 } \ev{ -x_3 + x_7} \rb \\ 
& = 2\textup{E}_3  (x_1 - x_3 + x_5 - x_7) \textup{E}_3\ev{ x_1 - x_5 } \textup{E}_3\ev{ -x_3 + x_7}  \\
& = 2\textup{E}_2  ( x_1 - x_3 ) \textup{E}_3\ev{ x_1 - x_5 } \textup{E}_3\ev{ -x_3 + x_7}  = f_0^{n}\ev{ \mathbf{x} }
\end{tblr}
$$

\item[b)] If $r=1$, then, by Lemma \ref{lem:helper_cycle}, the assumption $\mathbf{x} = \sigma_0\ev{ \mathbf{x}}$ implies that 
\begin{equation}
\label{eq:dct_subs_sigma}
\textup{E}_3\ev{ x_1 } = \textup{E}_3\ev{ x_7 } \qquad \text{and} \qquad \textup{E}_3\ev{ x_3 } = \textup{E}_3\ev{ x_5 }
\end{equation}
We calculate,
$$
\begin{tblr}{}
C^{3}\ev{ \mathbf{x} } - C^{3} \ev{ \sigma_{1} \ev{  \mathbf{x}  }}  \\
=  (x_1 x_2 x_3 + x_2 x_3 x_4 + x_3 x_4 x_5 + x_4 x_5 x_6 + x_5 x_6 x_7 + x_6 x_7 x_8 + x_7 x_8 x_1 + x_8 x_1 x_2) \\
- (x_3 x_4 x_1 + x_4 x_1 x_2 + x_1 x_2 x_7 + x_2 x_7 x_8 + x_7 x_8 x_5 + x_8 x_5 x_6 + x_5 x_6 x_3 + x_6 x_3 x_4) \\
&\hspace{-5cm}  \text{by Definition \ref{def:two_cycles}}  \\
=  (x_1^2 x_3 + x_1 x_3^2 + x_3^2 x_5 + x_3 x_5^2 + x_5^2 x_7 + x_5 x_7^2 + x_1 x_7^2 + x_7 x_1^2) \\
-( x_1 x_3^2 + x_1^2 x_3 + x_1^2 x_7 + x_1 x_7^2 +  x_5 x_7^2 + x_5^2 x_7 + x_3 x_5^2 + x_3^2 x_5) \\
&\hspace{-5cm} \text{by applying $\sigma_0$ on even indices} \\
=  (2x_1^2 x_5 + 2 x_1 x_5^2 + 2 x_5^3  + 2x_1^3) - (2 x_1^2 x_5 + 2 x_1 x_5^2 + 2 x_5^3  + 2 x_1^3) = 0 \\
&\hspace{-5cm}  \text{by Equations \ref{eq:dct_subs_sigma}} 
\end{tblr}$$

On the other hand, 
$$
\begin{tblr}{}
D^{3}\ev{ \mathbf{x} } - D^{3}\ev{ \sigma_{1} \ev{  \mathbf{x}  }} \\
=  (x_6 x_1 x_4  + x_7 x_2 x_5  + x_8 x_3 x_6 + x_1 x_4 x_7  + x_2 x_5 x_8  + x_3 x_6 x_1  + x_4 x_7 x_2  + x_5 x_8 x_3) \\
- (x_8 x_3 x_2  + x_5 x_4 x_7  + x_6 x_1 x_8  + x_3 x_2 x_5  + x_4 x_7 x_6  + x_1 x_8 x_3  + x_2 x_5 x_4  + x_7 x_6 x_1) \\
&\hspace{-5cm}  \text{by Definition \ref{def:two_cycles}} \\ 
=  (x_5 x_1 x_3  + x_7 x_1 x_5  + x_7 x_3 x_5  + x_1 x_3 x_7  + x_1 x_5 x_7  + x_3 x_5 x_1  + x_3 x_7 x_1  + x_5 x_7 x_3) \\
- (x_7 x_3 x_1  + x_5 x_3 x_7  + x_5 x_1 x_7 + x_3 x_1 x_5  +  x_3 x_7 x_5 + x_1 x_7 x_3  + x_1 x_5 x_3  + x_7 x_5 x_1) \\
&\hspace{-5cm}  \text{by applying $\sigma_0$ on even indices}\\
= \ev{ 4 x_1 x_5^2 + 4 x_1^2 x_5    } - \ev{ 4 x_1^2 x_5 + 4 x_1 x_5^2  } = 0 \\
&\hspace{-5cm}  \text{by Equations \ref{eq:dct_subs_sigma}}
\end{tblr}
$$
Therefore,
$$
\begin{tblr}{}
& G_n^{n}\ev{ \mathbf{x} } - G_n^{n}\ev{ \sigma_{1} \ev{  \mathbf{x}  } }  = E_3 \ev{ C^{3} \ev{ \mathbf{x} } - C^{3} \ev{ \sigma_{1} \ev{  \mathbf{x}  }  } } - E_3 \ev{ D^{3} \ev{ \mathbf{x} } - D^{3} \ev{ \sigma_{1} \ev{  \mathbf{x}  } } } \\ 
& = E_3\ev{ 0 } + E_3\ev{ 0 } = 0  
\end{tblr}
$$

\item[c)] If $r=2$, then, by Lemma \ref{lem:helper_cycle}, the assumption $\mathbf{x} = \sigma_0\ev{ \mathbf{x}} = \sigma_1\ev{ \mathbf{x} }$ implies that 
$$\textup{E}_3\ev{ x_1 } =  \textup{E}_3\ev{ x_2 } = \ldots = \textup{E}_3\ev{ x_8 } $$
Therefore, by Equation \ref{eq:G_n_defn} we obtain 
$$ 
G_n^{n}\ev{ \mathbf{x} } - G_n^{n}\ev{ \sigma_{2} \ev{  \mathbf{x}  }} = 0 .
$$
\end{enumerate}

\textbf{Case 2:} $2\leq k\leq n-1$. We apply induction on $r$. For the basis step, assume $r=0$. In particular, we have $k\neq n - r$. 

The graph $G_k^{n}$ is a disjoint union of all-even and all-odd edges. Formally, 
$$ G_k^{n} = p_0 \ev{ G_k^{n-1} } + p_1 \ev{ G_k^{n-1} } $$
Note that $\sigma_0$ maps even and odd edges to each other bijectively:
$$ 
\sigma_0 p_0 \ev{ G_k^{n-1} } = p_1 \ev{ G_k^{n-1} } \qquad \text{ and }\qquad \sigma_0 p_1 \ev{ G_k^{n-1} } = p_0 \ev{ G_k^{n-1} }.
$$ 
Therefore, 
$$
G_k^{n}\ev{\mathbf{x}} - G_k^{n}\ev{\sigma_0\ev{ \mathbf{x} }} = 0
$$

For the inductive step, let $1\leq r \leq n-2$ be given. First, we note that
\begin{align*}
& G_{k}^n\ev{ \mathbf{x} } - G_{k}^{n}\ev{ \sigma_{r} \ev{ \mathbf{x} } }  = G_{k}^n\ev{ \mathbf{x} } - \sigma_{r} G_{k}^{n}\ev{   \mathbf{x}  } \\
& = p_0  G_{k}^{n-1}\ev{ \mathbf{x} } + p_1  G_{k}^{n-1}\ev{ \mathbf{x} } - \sigma_{r}  p_0  G_{k}^{n-1}\ev{ \mathbf{x} } - \sigma_{r}  p_1  G_{k}^{n-1}\ev{ \mathbf{x} } \\
& = p_0  G_{k}^{n-1}\ev{ \mathbf{x} } + p_1  G_{k}^{n-1}\ev{ \mathbf{x} } - p_0  \sigma_{r-1}  G_{k}^{n-1}\ev{ \mathbf{x} } -    p_1  \sigma_{r-1}   G_{k}^{n-1}\ev{ \mathbf{x} }  &&\text{ by Remark \ref{rmk:interaction_sigma_p}}\\
& = p_0  \lb G_{k}^{n-1}\ev{ \mathbf{x} } - \sigma_{r-1} G_{k}^{n-1}\ev{ \mathbf{x} } \rb + p_1  \lb G_{k}^{n-1}\ev{ \mathbf{x} } - \sigma_{r-1} G_{k}^{n-1}\ev{ \mathbf{x} } \rb \\
& = p_0  \lb G_{k}^{n-1}\ev{ \mathbf{x} } -  G_{k}^{n-1}\ev{ \sigma_{r-1} \ev{  \mathbf{x} } } \rb + p_1  \lb G_{k}^{n-1}\ev{ \mathbf{x} } -  G_{k}^{n-1}\ev{ \sigma_{r-1} \ev{  \mathbf{x} } . } \rb 
\end{align*}

\textbf{Case 2.1:} $k = n-r$
\begin{align*}
& G_{n-r}^n\ev{ \mathbf{x} } - G_{n-r}^{n}\ev{ \sigma_{r} \ev{ \mathbf{x} } }  \\
& = p_0  \lb G_{n-r}^{n-1}\ev{ \mathbf{x} } -  G_{n-r}^{n-1}\ev{ \sigma_{r-1} \ev{  \mathbf{x} } } \rb + p_1  \lb G_{n-r}^{n-1}\ev{ \mathbf{x} } -  G_{n-r}^{n-1}\ev{ \sigma_{r-1} \ev{  \mathbf{x} }  } \rb \\
& = p_0  \ev{ f_{r-1}^{n-1}\ev{ \mathbf{x} } } + p_1 \ev{ f_{r-1}^{n-1}\ev{ \mathbf{x} } } && \text{ by inductive hypothesis} \\
& = f_{r}^{n}\ev{ \mathbf{x} } && \text{ by Lemma \ref{lem:even_odd_f}.}
\end{align*}

\textbf{Case 2.2:} $k\neq n-r$
\begin{align*}
& G_{k}^n\ev{ \mathbf{x} } - G_{k}^{n}\ev{ \sigma_{r} \ev{ \mathbf{x} } }  \\
& = p_0  \lb G_{k}^{n-1}\ev{ \mathbf{x} } -  G_{k}^{n-1}\ev{ \sigma_{r-1} \ev{  \mathbf{x} } } \rb + p_1  \lb G_{k}^{n-1}\ev{ \mathbf{x} } -  G_{k}^{n-1}\ev{ \sigma_{r-1} \ev{  \mathbf{x} }  } \rb \\
& = p_0 \ev{ 0 } + p_1 \ev{ 0 } = 0&& \text{ by inductive hypothesis.} 
\end{align*}
\end{proof}

%\begin{lemma}
%\label{lem:infinity_vertex}
%Let $0\leq r \leq n-2$ and $\mathbf{x} \in \C^{N+1}$ be given. Then, 
%$$ M^n_0 \ev{ \mathbf{x} } - M^n_0 \ev{ \sigma_r\ev{ \mathbf{x} } } = 0 $$
%\end{lemma}

\begin{lemma}
\label{lem:sigma_general_version}
Let $\mathbf{x}\in \C^{N+1}$ such that $\theta_n\ev{\mathbf{x}} = \mathbf{x}$ and $\mathbf{x} = \sigma_i\ev{ \mathbf{x} }$ for each $i=-1,0,\ldots, r-1$, where $0\leq r \leq n-2$ is fixed. Then,
$$ X^n\ev{ \mathbf{x} } - X^n\ev{ \sigma_{r}\ev{\mathbf{x}} }  = f_r^{n}\ev{ \mathbf{x} } $$
\end{lemma}
\begin{proof}
First, we show that
$$M^n_0 \ev{ \mathbf{x} } - M^n_0 \ev{ \sigma_r\ev{ \mathbf{x} } } = 0$$
Recall that 
$$ M^n_0  = x_0 \textup{E}_2 \ev{ x_1 x_2 + x_3 x_4}$$
By Remark \ref{rmk:sigma_exceeds_threshold}, we have $ \text{E}^{n}_2  \ev{ x_i } =  \text{E}^{n}_2  \ev{ \sigma_t\ev{ x_i } } $ whenever $t\geq 2$. Hence, we only need to consider $r = 0,1$: 
\begin{align*}
& M^n_0 \ev{ \mathbf{x} } - M^n_0 \ev{ \sigma_0\ev{ \mathbf{x} } } = x_0  \textup{E}_2 \ev{ x_1 x_2 + x_3 x_4}  - x_0  \textup{E}_2 \ev{ x_2 x_1 + x_4 x_3}  = 0  \\
& M^n_0 \ev{ \mathbf{x} } - M^n_0 \ev{ \sigma_1\ev{ \mathbf{x} } } = x_0  \textup{E}_2 \ev{ x_1 x_2 + x_3 x_4}  - x_0  \textup{E}_2 \ev{ x_3 x_4 + x_1 x_2}  = 0  
\end{align*}
Therefore,
\begin{align*}
& X^n\ev{ \mathbf{x} } - X^n\ev{ \sigma_{r}\ev{\mathbf{x}} }  = M^n_0 \ev{ \mathbf{x} } - M^n_0 \ev{ \sigma_{r}\ev{\mathbf{x}} } + \sum_{k = 2}^{n} G_k^{n} \ev{ \mathbf{x} } - G_k^{n} \ev{  \sigma_{r}\ev{\mathbf{x}} }  \\
& = 0 + G_{n-r}^{n} \ev{ \mathbf{x} } - G_{n-r}^{n} \ev{  \sigma_{r}\ev{\mathbf{x}} }  = f_r^{n}\ev{ \mathbf{x} } & \hspace{-3cm}\text{by Lemma \ref{lem:induction_cycles} } 
\end{align*}
\end{proof}

Next, we compare the neighborhoods of $1$ and $1+2^r$, for each $r=0,1,\ldots, n-3$. First, we have some definitions and helper lemmas:
\begin{defn}
Let $\mathcal{H} = \ev{[n],E\ev{ \mathcal{H} } }$ be a hypergraph, of rank $m\geq 2$. For each $i\in [n]$, the \textit{link} of $i$ is defined as the partial derivative of the Lagrangian of $\mathcal{H}$:
$$  \mathcal{H}[i] = \dfrac{\partial}{ \partial x_i } \mathcal{H}\ev{ \mathbf{x} } $$
Note that the link $ \mathcal{H}[i]$ is a hypergraph of rank $m-1$, for each $i=1,\ldots,n$. Let $\text{deg}_\mathcal{H}\ev{i}$ be the number of edges in $\mathcal{H}[i]$. In other words, $\textup{deg}_\mathcal{H}\ev{i} = \mathcal{H}[i]\ev{ \mathbbm{1} }$. Taking the link of a vertex is a commutative operation:
$$ \mathcal{H}[i,j] \underset{\textup{def}}{=} \mathcal{H}[i][j] = \mathcal{H}[j][i] \text{ for each pair } i\neq j $$
Let $\textup{codeg}_\mathcal{H}\ev{i,j}$ be the number of edges in $\mathcal{H}[i,j]$. Equivalently, $\textup{codeg}_\mathcal{H}\ev{i,j} = \mathcal{H}[i,j]\ev{ \mathbbm{1} }$.
\end{defn}

\begin{lemma}
\label{lem:helper_parity}
Let $i\in V_{n-1}$ and $1\leq r \leq n-3$ be fixed. Let $\bm{\epsilon}\in \{0,1\}^{r}$ be a zero-one vector. Then, we have 
\begin{enumerate}
\item $p^{n}_{\bm{\epsilon}}\ev{ i } = 1  \text{ if and only if } i \equiv 1 \pmod{2^{n-r}}\text{ and } \bm{\epsilon} = \mathbbm{1}_r$
\item $p^{n}_{\bm{\epsilon}}\ev{ i } = 2  \text{ if and only if } i \equiv 1 \pmod{2^{n-r}} \text{ and } \bm{\epsilon} = \ev{ 0, \smalloverbrace{1,\ldots,1}^{\text{\clap{$r-1$  many}}}} $
\end{enumerate}

\end{lemma}
\begin{proof}
Let $ i = \sum_{j=0}^{n-1} 2^j b_j(i) $ be the binary representation of $i$, for each $i\in V_{n-1}$. By Remark \ref{rmk:repeated_app}, we have
$$ p_{\bm{\epsilon}}(i) = 2^{r} i  - \sum_{j = 0}^{r-1} \epsilon_j 2^{j} 
$$
In other words, 
$$ 
\sum_{j=r}^{n-1+r} 2^{j} b_j(i)  = p_{\bm{\epsilon}}(i) + \sum_{j = 0}^{r-1} \epsilon_j 2^{j} 
$$

\begin{enumerate}
\item Assume $p_{\bm{\epsilon}}\ev{ i } = 1$. The equation above implies that $\epsilon_0 = \epsilon_1 = \ldots = \epsilon_{r-1} = 1$. Then, we obtain $ 2^{r} i  = 2^{r} \pmod{2^n}$, which implies $i \equiv 1 \pmod{2^{n-r}}$.

Conversely, for each $1\leq s\leq 2^{n-r}$, we have
$$ 
p_{\mathbbm{1}} \ev{ 1 + s \cdot 2^{n-r}} = 2^{r} \ev{ 1 + s \cdot 2^{n-r}} - \ev{ 2^{r} - 1 } = 1 + s \cdot 2^n \equiv 1 \pmod{2^n} 
$$

\item Assume $p_{\bm{\epsilon}}\ev{ i } = 2$. In this case, we obtain $\epsilon_0=0$ and $ \epsilon_1 = \ldots = \epsilon_{r-1} = 1$. Then, we obtain $ 2^{r} i  = 2^{r} \pmod{2^n}$, which implies $i \equiv 1 \pmod{2^{n-r}}$. 

Conversely, for each $1\leq s\leq 2^{n-r}$, we get
\begin{align*}
&p_{\bm{\epsilon}} \ev{ 1 + s \cdot 2^{n-r} } = p_0 p_{\mathbbm{1}_{r-1}} \ev{ 1 + s \cdot 2^{n-r}} = p_0 \ev{ 2^{r-1} \ev{ 1 + s \cdot 2^{n-r}} - \ev{ 2^{r-1} - 1 } } \\
& = p_0 \ev{ 1 + s \cdot 2^{n-1} } = 2+s\cdot 2^n \equiv 2 \pmod{2^n} 
\end{align*}
\end{enumerate}
\end{proof}

\begin{lemma}
\label{lem:induction_neigh}
Let $0\leq r \leq n-2$ and $2\leq k \leq n$ be fixed. Let $\mathbf{x}\in \C^{N+1}$ such that $\theta_n\ev{\mathbf{x}} = \mathbf{x}$ and $\mathbf{x} = \sigma_i\ev{ \mathbf{x} }$ for each $i=0,\ldots, r$, where $r\geq 0$. Then,
$$ G_k^{n}[1]\ev{ \mathbf{x} } - G_k^{n}[1+2^r]\ev{ \mathbf{x} } = \begin{cases} \lb \textup{E}_{r+3} \ev{  x_1 - x_{ 1 + 2^{r+2} } } \rb^2 & \text{ if } k = n-r \\ 0 & \text{ if }  k\neq n-r
\end{cases} $$
\end{lemma}
\begin{proof}
\textbf{Case 1:} $k=n$. In this case, we would like to show that
$$G^{n}_n[1] \ev{ \mathbf{x} } - G^{n}_n[1+2^r] \ev{ \mathbf{x} } = \begin{cases} \lb \textup{E}_{3} \ev{  x_1 - x_{ 5 } } \rb^2  & \text{ if } r=0 \\ 0 & \text{ if } 1\leq r \end{cases}$$

\begin{enumerate}
\item[a)] First, assume $r=0$. By Lemma \ref{lem:helper_cycle}, Part 2, we have
\begin{equation}
\label{eq:to_be_used}
\textup{E}_3\ev{ x_7 } = \textup{E}_3\ev{ x_1 } \qquad \text{and} \qquad \textup{E}_3\ev{ x_3 } = \textup{E}_3\ev{ x_5 }
\end{equation}

By Definition \ref{def:cycles}, we have
$$
\begin{tblr}{}
G^{n}_n[1] \ev{ \mathbf{x} } -  G^{n}_n[2] \ev{\mathbf{x} } \\
=  \textup{E}_{3} \ev{x_7 x_8 + x_8 x_2 + x_2 x_3 + x_6  x_3 + x_4 x_7 + x_4 x_6} \\
\qquad - \textup{E}_{3} \ev{x_8  x_1 + x_1 x_3 + x_3 x_4 + x_7 x_4 + x_5 x_8 + x_5 x_7}  \\
= \textup{E}_{3} \ev{x_7 x_7 + x_7 x_1 + x_1 x_3 + x_5  x_3 + x_3 x_7 + x_3 x_5} \\
\qquad - \textup{E}_{3} \ev{x_7  x_1 + x_1 x_3 + x_3 x_3 + x_7 x_3 + x_5 x_7 + x_5 x_7} \\
&\hspace{-4cm} \text{by $\sigma_0$ applied to even indices}\\
= \textup{E}_{3} \ev{x_7 x_7 + 2 x_5 x_3 } - \textup{E}_{3} \ev{ x_3 x_3 + 2 x_5 x_7 } \\
= \textup{E}_{3} \ev{x_1^2 + 2 x_5^2 } - \textup{E}_{3} \ev{ x_5^2 + 2 x_5 x_1 } & \hspace{-3cm}\text{by Equations \ref{eq:to_be_used}} \\ 
= \lb \textup{E}_3 \ev{  x_1-x_5  } \rb^{2} & \hspace{-4cm}\text{as $\textup{E}_3$ is a ring homomorphism}
\end{tblr} $$

\item[b)] Now, we assume $1\leq r \leq n-2$. By Lemma \ref{lem:helper_cycle},
$$ \textup{E}_3\ev{ x_i } = \textup{E}_3\ev{ x_j }$$
for any $1\leq i < j\leq 8$. Therefore, by Definition \ref{def:cycles}, we have
$$ 
G_{n}^n[1] = G_n^{n}[1+2^r] = 0.
$$
\end{enumerate}

\noindent \textbf{Case 2:} $2\leq k \leq n-1$. Now, we apply induction on $r$. For the base case, assume $r=0$. In particular, we have $k\neq n - r$. In this case, we would like to show that 
$$ G^{n}_k[1] \ev{ \mathbf{x} } - G^{n}_k[2] \ev{ \mathbf{x} } = 0 $$

\noindent \textbf{Case 2a:} $k = 2$

We describe the neighborhoods of $1$ and $2$. Let $e\in E\ev{ G_{2}^{n} }$ be an edge. By Remark \ref{rmk:rec_defn} Part 1, we have $\mathbf{x}^{e} = p_{\bm{\epsilon}} \ev{ x_i x_{i+2} x_{i+4} }$, for some $\bm{\epsilon}\in \{0,1\}^{n-3}$ and some $1\leq i\leq 4$.  If $1\in e$, then, by Lemma \ref{lem:helper_parity}, we have $\bm{\epsilon} = \mathbbm{1}_{n-3}$ and $i = 1$. Hence,
$$ 
\mathbf{x}^{e} = p_{\mathbbm{1}_{n-3}} \ev{ x_1 x_3 x_5 } = x_1 x_{3\cdot 2^{n-3} - \ev{ 2^{n-3} - 1}} x_{5\cdot 2^{n-3} - \ev{ 2^{n-3} - 1}} = x_1 x_{2^{n-2}+1} x_{2^{n-1}+1} 
$$
If $2\in e$, then, by Lemma \ref{lem:helper_parity}, we have $\bm{\epsilon} = \ev{ 0, \smalloverbrace{1,\ldots,1}^{\text{\clap{$n-4$  many}}}}$ and $i = 1$. Hence,
$$ 
\mathbf{x}^{e}  = p_0 p_{ \mathbbm{1}_{n-4} }  \ev{ x_1 x_3 x_5 } = p_0 \ev{ x_1 x_{2^{n-3}+1} x_{2^{n-2}+1} } = x_2 x_{ 2^{n-2} +2 } x_{2^{n-1} +2} 
$$
Using the assumption that $\sigma_0\ev{\mathbf{x}} = \mathbf{x}$, we conclude, 
$$ G^{n}_2[1] \ev{ \mathbf{x} } - G^{n}_2[2] \ev{ \mathbf{x} } = x_{2^{n-2}+1} x_{2^{n-1}+1} -  x_{2^{n-2}+2} x_{2^{n-1}+2} = 0$$

\noindent \textbf{Case 2b:} $3\leq k \leq n-1$

Let $e\in E\ev{ G_{k}^{n} }$ be an edge. By Remark \ref{rmk:rec_defn} Part 2, we have 
$$ \mathbf{x}^{e} = p_{ \bm{\epsilon} } \ev{ x_{j_1} x_{j_2} x_{j_3} } $$
for some edge $\{j_1,j_2,j_3\}\in E\ev{ G_k^k }$ with $j_1<j_2<j_3$ and some zero-one vector $\bm{\epsilon} \in \{0,1\}^{ n-k } $.

\begin{itemize}
\item If $1\in e$, then, by Lemma \ref{lem:helper_parity}, we have $\bm{\epsilon} = \mathbbm{1}_{n-k}$ and $ j_1 \equiv 1 \pmod{2^k}$. Since $\{j_1,j_2,j_3\}\in E\ev{ G_k^k }$, it follows that $j_1 = 1$. Furthermore, $\{1,j_2,j_3\}\in E\ev{ G_k^k }$ implies that 
$$ 
j_2 = i_2 + 8s_2 \pmod{V_k} \qquad \text{ and } \qquad j_3 = i_3 + 8s_3 \pmod{V_k}
$$
where $1\leq i_2,i_3 \leq 8$, $1\leq s_2,s_3\leq 2^{k-3}$ and $\{1,i_2,i_3\}\in G_3^{3}$. Hence,
$$ \mathbf{x}^{e} =  p_{\mathbbm{1}_{n-k}} \ev{ x_1 x_{i_2 + 8s_2} x_{i_3 + 8s_3}} $$ 
For each $i$ and $s$, we have 
$$ p_{ \mathbbm{1}_{n-k} } \ev{i + 8s } =  2^{n-k}\ev{ i + 8s } - \ev{ 2^{n-k} - 1 } = 2^{n-k}\ev{ i + 8 s -1 } + 1$$

Therefore,
\begin{align*}
& G_{k}^{n}[1] \ev{\mathbf{x}}  = \sum_{ \{1,i_2,i_3\}\in E\ev{ G_3^3 } } \sum_{   1\leq s_2,s_3\leq 2^{k-3}}  x_{ 2^{n-k}\ev{ i_2 + 8 s_2 -1 } + 1 }  x_{ 2^{n-k}\ev{ i_3 + 8 s_3 -1 } + 1 } 
\end{align*}

\item Similarly, if $2\in e$, then, by Lemma \ref{lem:helper_parity}, we have $\bm{\epsilon} = \ev{ 0, \smalloverbrace{1,\ldots,1}^{\text{\clap{$n-k-1$  many}}}}$ and $ j_1 \equiv 1 \pmod{2^k}$. This implies, as before, that
$$ \mathbf{x}^{e} = p_{\bm{\epsilon} } \ev{ x_1 x_{i_2 + 8s_2} x_{i_3 + 8s_3} } $$
for some $1\leq i_2,i_3 \leq 8$ and some $1\leq s_2,s_3\leq 2^{k-3}$, with $\{1,i_2,i_3\}\in G_3^{3}$. For each $i$ and $s$, we have 
$$ p_{ \bm{\epsilon} } \ev{ i + 8s } =  2^{n-k}\ev{ i + 8s } - \ev{ 2^{n-k} - 2 }= 2^{n-k}\ev{ i + 8 s -1 } + 2$$

Therefore,
\begin{align*}
& G_{k}^{n}[2] \ev{\mathbf{x}}  = \sum_{ \{1,i_2,i_3\}\in E\ev{ G_3^3 } } \sum_{   1\leq s_2,s_3\leq 2^{k-3}}  x_{ 2^{n-k}\ev{ i_2 + 8 s_2 -1 } + 2 }  x_{ 2^{n-k}\ev{ i_3 + 8 s_3 -1 } + 2 }
%& = \sum_{x_s x_t \in G_{k}^{n}[1]} x_{s+1} x_{t+1}
\end{align*}
\end{itemize}

Using the assumption that $\sigma_0\ev{\mathbf{x}} = \mathbf{x}$, we obtain 
$$ 
G^{n}_k[2] \ev{ \mathbf{x} } = \sigma_0\ev{ G_{k}^{n}[1] \ev{\mathbf{x}} } = G^{n}_k[1] \ev{ \sigma_0\ev{ \mathbf{x} } } =  G^{n}_k[1] \ev{ \mathbf{x} }
$$
For the inductive step, assume $1\leq r \leq n-2$. First, note that
\begin{align*}
& G_{k}^n[1] \ev{ \mathbf{x} } - G_k^{n}[1+2^r]\ev{ \mathbf{x} }  \\
& = p_1  G_{k}^{n-1}[1] \ev{ \mathbf{x} } -  p_1  G_k^{n-1}[1+2^{r-1}] \ev{ \mathbf{x} } \\
& = p_1  \ew{ G_{k}^{n-1}[1] \ev{ \mathbf{x} } - G_k^{n-1}[1+2^{r-1}] \ev{ \mathbf{x} } }  
\end{align*}
If $k = n-r$, then
\begin{align*}
& G_{k}^n[1] \ev{ \mathbf{x} } - G_k^{n}[1+2^r]\ev{ \mathbf{x} }  \\
& = p_1  \lb \textup{E}_{r+2} \ev{  x_1 - x_{ 1 + 2^{r+1} } } \rb^2 && \text{ by the inductive hypothesis} \\ 
& = \lb \textup{E}_{r+3} \ev{  x_1 - x_{ 1 + 2^{r+2} } } \rb^2 && \text{ by Lemma \ref{lem:odd_even_eight}.}
\end{align*}
If $k\neq n-r$, then
\begin{align*}
& G_{k}^n[1] \ev{ \mathbf{x} } - G_k^{n}[1+2^r]\ev{ \mathbf{x} } = p_1 \ev{ 0 } = 0 && \text{ by the inductive hypothesis.} 
\end{align*}
\end{proof}

\begin{lemma}
\label{lem:neigh_general_version}
Let $\mathbf{x}\in \C^{N+1}$ such that $\theta_n\ev{\mathbf{x}} = \mathbf{x}$ and $\mathbf{x} = \sigma_i\ev{ \mathbf{x} }$ for each $i=0,\ldots, r$, where $0 \leq r \leq n-3$. Then,
\begin{align*}
& X^n[1] \ev{ \mathbf{x} } - X^n[1+2^r] \ev{ \mathbf{x} } = \lb \textup{E}_{r+3} \ev{  x_1 - x_{ 1 + 2^{r+2} } } \rb^2
\end{align*}
\end{lemma}
\begin{proof}
First, we show that 
$$ M^n_0[1] \ev{ \mathbf{x} } - M^n_0[1+2^r] \ev{ \mathbf{x} } = 0 $$
Recall that
$$ M^n_0 = x_0 \textup{E}_2 \ev{ x_1 x_2 + x_3 x_4}$$
Hence, we have
\begin{align*}
& M^n_0[1] = x_0 \textup{E}_2 \ev{ x_2 } \\ 
& M^n_0[2] = x_0 \textup{E}_2 \ev{ x_1 } \\ 
& M^n_0[3] = x_0 \textup{E}_2 \ev{ x_4 }  
\end{align*}
By Remark \ref{rmk:sigma_exceeds_threshold}, we need only to consider $r = 0,1$: 
\begin{align*}
& M^n_0[1] \ev{ \mathbf{x} } - M^n_0[2] \ev{ \mathbf{x} }  = x_0 ( \textup{E}_2 \ev{ x_2 } - \textup{E}_2 \ev{ x_1 } ) = 0 & \text{ by } \sigma_0 \ev{ \mathbf{x} } = \mathbf{x} \\
& M^n_0[1] \ev{ \mathbf{x} } - M^n_0[3] \ev{ \mathbf{x} }  = x_0 ( \textup{E}_2 \ev{ x_2 } - \textup{E}_2 \ev{ x_4 } ) = 0 & \text{ by } \sigma_1 \ev{ \mathbf{x} } = \mathbf{x} 
\end{align*}
which shows that $M^n_0[1]\ev{ \mathbf{x} } - M^n_0[1+2^r]\ev{ \mathbf{x} } = 0 $. Therefore,
\begin{align*}
& X^n[1] \ev{ \mathbf{x} } - X^n[1+2^r] \ev{ \mathbf{x} } =  M^n_0[1]  \ev{ \mathbf{x} } - M^n_0[1+2^r] \ev{ \mathbf{x} } + \sum_{k = 2}^{n} G_k^{n}[1]  \ev{ \mathbf{x} } - G_k^{n}[1+2^r]  \ev{  \mathbf{x} }  \\ 
& = 0 + G_{n-r}^{n}[1] \ev{ \mathbf{x} } - G_{n-r}^{n}[1+2^r] \ev{  \sigma_{r}\ev{\mathbf{x}} }  = \lb \textup{E}_{r+3} \ev{  x_1 - x_{ 1 + 2^{r+2} } } \rb^2 & \hspace{-4cm}\text{by Lemma \ref{lem:induction_neigh}. } 
\end{align*}

\end{proof}

In \cite{kocay}, the degree of each vertex in $\Gamma_n$ is calculated: 
\begin{lemma}
\label{lem:kocay_degrees}
$$ \textup{deg}_{\Gamma_n} \ev{ i } = \begin{cases} 2^{2n-3}-1 & \text{if } 1\leq i \leq 2^{n-2} \text{ or } 1 + 3\cdot 2^{n-2} \leq i \leq 2^{n} \\ 2^{2n-3} & \text{ otherwise } \end{cases}$$
\end{lemma}
In particular, we note that $\Gamma_n$ is not regular.

\begin{defn}
Let $\sim$ denote the projective equivalence relation, i.e., two vectors satisfy $\mathbf{x} \sim \mathbf{y}$ if and only if $\mathbf{x} = a \mathbf{y}$ for some $a\in \R\setminus\{0\}$. 
\end{defn}

The following lemma may have independent interest, and is therefore stated for hypergraphs of any rank.

\begin{lemma}
\label{lem:regular_extended}
Let $\mathcal{G} = \ev{[n],E\ev{\mathcal{G}} } $ be a hypergraph, of rank $m\geq 2$. Let $\mathcal{H} = \ev{[n] \cup \{0\}, E\ev{\mathcal{H}} }$ be obtained from $\mathcal{G}$ by appending a vertex $0 \notin [n]$ with $\textup{codeg}_\mathcal{H}\ev{0,i} = \textup{codeg}_\mathcal{H}\ev{0,j}$, for each $i,j\in [n]$. Let $\ev{\lambda_\mathcal{G},\mathbf{v}}$ and $\ev{\lambda_\mathcal{H}, \mathbf{w}}$ be the principal eigenpairs of $\mathcal{G}$ and $\mathcal{H}$, respectively, where $\mathbf{w} =  \ev{ w_0,w_1,\ldots,w_n }$. Then, the following are equivalent:
\begin{enumerate}
\item[i.] $\mathcal{G}$ is regular. 
\item[ii.] $\mathbf{v} \sim \mathbbm{1} $. 
\item[iii.] $\mathbf{w} \sim \widetilde{\mathbf{w}} := \ev{ \widetilde{w}_0, 1,\ldots, 1 } \text{ for some $\widetilde{w}_0 \in \R\setminus\{0\}$ }$.
\end{enumerate}
\end{lemma}

\begin{proof}

\phantom{a}

\begin{enumerate}
\item $(i) \Longrightarrow (ii)$: Follows from \cite[Theorem 3.8]{cd1} and the uniqueness of the principal eigenpair.

$(i) \Longleftarrow (ii)$: Note that $\deg_\mathcal{G}\ev{i} = \mathcal{G}[i] \ev{\mathbbm{1}} = \mathcal{G}[j]\ev{\mathbbm{1}} = \deg_\mathcal{G}\ev{j}$ for each $i,j$.

\item Let $\gamma = \textup{codeg}_\mathcal{H}\ev{0,i} = \mathcal{H}[0,i] \ev{\mathbbm{1}}$ for each $i\in [n]$. 
Consider the link of $i = 1,\ldots, n$ :
\begin{align}
\label{eq:eigenv_equation}
\mathcal{H}[i] \ev{ \mathbf{x} } = \mathcal{G}[i] \ev{ \mathbf{x} } + x_0 \mathcal{H}[0,i] \ev{ \mathbf{x} }
\end{align}
$(iii) \Longrightarrow (ii)$: By assumption, $\widetilde{\mathbf{w}}$ is an eigenvector of $\mathcal{H}$. So, we have
$$ \mathcal{H}[i] \ev{ \widetilde{\mathbf{w}} } = \lambda_\mathcal{H} \cdot 1^{m-1} = \lambda_\mathcal{H}$$
for each $i=1,\ldots,n$. We evaluate Equation \ref{eq:eigenv_equation} at $\widetilde{\mathbf{w}}$ to obtain 
$$
\lambda_\mathcal{H} = \mathcal{H}[i] \ev{ \widetilde{\mathbf{w}} } = \mathcal{G}[i] \ev{ \mathbbm{1}  } + \widetilde{w}_0 \mathcal{H}[0,i] \ev{\mathbbm{1}} = \mathcal{G}[i] \ev{ \mathbbm{1}  } + \widetilde{w}_0 \gamma  
$$
for each $i$. Therefore, $\dfrac{1}{\sqrt[m]{n}} \cdot \mathbbm{1}$ is the principal eigenvector of $\mathcal{G}$.

$(ii) \Longleftarrow (iii)$: By assumption, the following equation is satisfied:
$$ 
\mathcal{G}[i] \ev{ \mathbbm{1} } = \lambda_\mathcal{G} \cdot 1^{m-1} = \lambda_\mathcal{G} \qquad \text{ for each } i = 1,\ldots, n
$$
Consider the vector $\mathbf{u} = \ev{ u,1,\ldots,1 }$. We would like to show that there are some constants $u,\lambda_u\in \R_{>0}$ such that the pair $\ev{\lambda_u,\mathbf{u}}$ satisfies the eigenvalue equations for $\mathcal{H}$. When we evaluate Equation \ref{eq:eigenv_equation} at $\mathbf{u}$, we get 
$$ 
\mathcal{H}[i] \ev{ \mathbf{u} } = \mathcal{G}[i] \ev{ \mathbbm{1} } + u \cdot \textup{codeg}_\mathcal{H}\ev{0,i} = \lambda_\mathcal{G} + u \gamma 
$$
for each $i=1,\ldots,n$. Let $\lambda_u : = \lambda_\mathcal{G} + u \gamma $. In particular, $\lambda_u \cdot 1^{m-1} = \mathcal{H}[i] \ev{ \mathbf{u} }$ for $i=1,\ldots,n$.

Also, we have
$$ \mathcal{H}[0] \ev{ \mathbf{u} } =  \deg_\mathcal{H}\ev{0} $$
which is a constant. It is enough to show that $\mathcal{H}[0] \ev{ \mathbf{u} } = \lambda_u u^{m-1}$ for some $u \in \R_{>0}$. In other words, it is enough to show that the equation 
$$\deg_\mathcal{H}\ev{0} = \ev{ \lambda_\mathcal{G} + u \gamma } \cdot u^{m-1}  =  \lambda_\mathcal{G} u^{m-1} + \gamma u^{m}$$
has a solution, which follows from the Intermediate Value Theorem. 
\end{enumerate}
\end{proof}

\begin{lemma}
\label{lem:main_lemma}
Let $n\geq 3$ be fixed. Let $\mathbf{x}$ be the principal eigenvector of $X^n$. Then, $\textup{E}_2^{n}\ev{ x_1 - x_3 } \neq 0$.
\end{lemma}

\begin{proof}
Suppose, for a contradiction, that $\textup{E}_2\ev{ x_1 - x_3 } = 0$. By induction on $r$, we show that 
\begin{equation}
\label{eq:regularity}
\sigma_{r} \ev{ \mathbf{x} } = \mathbf{x} \qquad \text{ for each $r =  0,\ldots, n-2$.}
\end{equation}
For the base case, let $r = 0$. By Lemma \ref{lem:sigma_general_version}, we have 
$$X^n\ev{ \mathbf{x} } - X^n\ev{ \sigma_{0}\ev{\mathbf{x}} }  = 2 \cdot \textup{E}_{2} \ev{  x_1 - x_{ 3 } } \cdot \textup{E}_{3} \ev{ x_{1} - x_{ 5 } }   \cdot  \textup{E}_{3} \ev{  - x_{ 3  } + x_{ 7 } } = 0$$
By the uniqueness of the principal eigenpair, we obtain $\mathbf{x} =\sigma_0\ev{ \mathbf{x} }$. 

Now, let $1\leq r \leq n-3$ be fixed. By the inductive hypothesis, we have
$$ \sigma_{i} \ev{ \mathbf{x} } = \mathbf{x} $$
for each $i=0,\ldots,r-1$. By Lemma \ref{lem:neigh_general_version}, we have
$$X^n[1] \ev{ \mathbf{x} } - X^n[1+2^{r-1}] \ev{ \mathbf{x} } = \lb \textup{E}_{r+2} \ev{  x_1 - x_{ 1 + 2^{r+1} } } \rb^2    $$
Since $\sigma_{r-1} \ev{ \mathbf{x} } = \mathbf{x}$, it follows that $x_1 = x_{1+2^{r-1}} $ and so, $X^n[1] \ev{ \mathbf{x} } - X^n[1+2^{r-1}] \ev{ \mathbf{x} } = 0$. Hence, we have $\textup{E}_{r+2} \ev{  x_1 - x_{ 1 + 2^{r+1} } } = 0$. 

Using Lemma \ref{lem:sigma_general_version}, we obtain
$$X^n\ev{ \mathbf{x} } - X^n\ev{ \sigma_{r}\ev{ \mathbf{x} } }  = f^{n}_{r} \ev{ \mathbf{x} } $$
As $ f^{n}_{r} \ev{ \mathbf{x} } $ is divisible by $\textup{E}_{r+2} \ev{  x_1 - x_{ 1 + 2^{r+1} } }$, we obtain $X^n\ev{ \mathbf{x} } - X^n\ev{ \sigma_{r}\ev{ \mathbf{x} } } = 0$. By the uniqueness of the principal eigenpair, we conclude $\mathbf{x} =\sigma_r\ev{ \mathbf{x} }$, which finishes the induction. 

Equation \ref{eq:regularity} implies $x_i = x_j$, for each $1\leq i < j \leq 2^{n-1}$. Since $\theta_n\ev{ \mathbf{x} } = \mathbf{x}$, we obtain
$$ x_i = x_j \qquad \text{ for each } 1\leq i < j \leq 2^{n}$$
By Lemma \ref{lem:regular_extended}, we infer that $\Gamma_n$ is a regular hypergraph, which contradicts Lemma \ref{lem:kocay_degrees}.
\end{proof}

We may now prove our main theorem.

\begin{theorem}
The principal eigenvalues of $X^n$ and $Y^n$ are different. 
\end{theorem}
\begin{proof}
Let $\ev{ \lambda, \mathbf{x} }, \ev{ \mu, \mathbf{y}}$ be the principal eigenpairs of $X^n$ and $Y^n$, respectively. 

In particular, we have 
$$ \lambda = X^n \ev{\mathbf{x}}   \quad \text{ and }\quad    \mu = Y^n \ev{\mathbf{y}}  $$
By Lemma \ref{lem:basis_step}, we have
\begin{align*}
& Y^n \ev{ \mathbf{x} } - X^n\ev{ \mathbf{x} } = x_0\lb \textup{E}_2\ev{ x_1 - x_3 } \rb^{2} 
\end{align*}

Therefore,
\begin{align*}
&  \mu - \lambda  =  Y^n \ev{\mathbf{y}} - X^n \ev{\mathbf{x}}   \\
& = \ev{Y^n \ev{\mathbf{y}} - Y^n \ev{\mathbf{x}}} + \ev{Y^n \ev{\mathbf{x}} - X^n \ev{\mathbf{x}}}  \\
& = \ev{Y^n \ev{\mathbf{y}} - Y^n \ev{\mathbf{x}}} + x_0\lb \textup{E}_2\ev{ x_1 - x_3 } \rb^{2} > 0  & \text{ by Lemma \ref{lem:main_lemma}}
\end{align*}
In other words, $\mu > \lambda$.
\end{proof}

\begin{corollary}
The infinite pair of hypergraphs $\{ \ev{X^n, Y^n}\}_{n\geq 3}$ are hypomorphic, but their characteristic polynomials are different. 
\end{corollary}

\section{Acknowledgement}

Thanks to Alexander Farrugia for helpful discussions that motivated much of the present work, and to the anonymous referee for their careful reading and suggestions. SageMath (\cite{sagemath}) was integral to the development of our arguments; the second author maintains a repository of SageMath code (\cite{okur}) that can be used to perform hypergraph calculations, such as the ones that appear above.

\bibliographystyle{plain}
\bibliography{bibliography}

\end{document}